\documentclass[a4paper,10pt]{amsart}

\usepackage[american]{babel}
\usepackage{cite}
\usepackage{amsfonts}
\usepackage{latexsym}
\usepackage{amsthm}
\usepackage{amsopn}
\usepackage{amsmath}
\usepackage[all]{xy}
\usepackage{verbatim}
\usepackage{amssymb}
\usepackage{mathrsfs}
\theoremstyle{remark}
\newtheorem{remark}{Remark}
\newtheorem{notations}[remark]{Notations}
\newtheorem{examples}[remark]{Examples}

\theoremstyle{plain}
\newtheorem{theorem}[remark]{Theorem}
\newtheorem{proposition}[remark]{Proposition}
\newtheorem{lemma}[remark]{Lemma}
\newtheorem{corollary}[remark]{Corollary}

\theoremstyle{definition}
\newtheorem{definition}[remark]{Definition}

\newcommand{\lineq}{\equiv_{\mathrm{lin}}}
\newcommand{\numeq}{\equiv_{\mathrm{num}}}
\begin{document}
\title[Surfaces with nontrivial surjective
endomorphisms]{Surfaces with nontrivial surjective endomorphisms of any
given degree}

%\date{\today}

\author{Antonio Rapagnetta}
\address{Via Della Ricerca Scientifica 1, 00133 - Roma (Italy)}
\email{rapagnet@axp.mat.uniroma2.it}

\author{Pietro Sabatino}
\address{Via Delle Mimose 9 Int. 14, 00172 - Roma (Italy)}
\email{pietrsabat@gmail.com}

\subjclass[2010]{Primary 14J25; Secondary 14J26, 14A10}
\keywords{nontrivial surjective endomorphism, projective bundles, \'etale
quotients}
\thanks{The authors would like to thank Massimiliano Mella who suggested this
problem to the second author many years ago.}

\begin{abstract}
  We present a complete classification of complex projective surfaces $X$ with
  nontrivial self-maps (i.e. surjective morphisms $f:X\rightarrow X$ which are
  not isomorphisms) of any given degree. The starting point of our
  classification
  are results contained in \cite{fujimoto:endo} and \cite{nakayama:ruled} that
  provide a list of surfaces that admit at
  least one nontrivial self-map. We then proceed
  by a
  case by case analysis that blends geometrical and arithmetical arguments in
  order to exclude that certain prime numbers appear as degrees of nontrivial
  self-maps of certain surfaces.
\end{abstract}

\maketitle
\section{Introduction and statement of the result}
We work over the complex field $\mathbb{C}$. All the varieties that we will
consider will be projective and smooth, in case of exceptions
we will explicitly state it.

\begin{definition}
  Let $X$ be a complex projective variety and let $f:X\rightarrow X$ be a
  surjective endomorphism of $X$, i.e. a morphism of $X$ onto itself.  $f$
  is said to be a \emph{nontrivial self-map} if it is not an isomorphism,
  or equivalently if
  the degree of $f$ is greater than or equal to $2$.
\end{definition}

In what follows we will provide a complete classification of surfaces that
admit a nontrivial self-map of any given degree, this is the content of the
following Theorem.

\begin{theorem}
  \label{theorem:principal}
  A surface admits nontrivial self-maps of any given degree
  if and only if it is one of the following:
  \begin{itemize}
    \item[(i)] $\mathbb{P}_1\times\mathbb{P}_1$
    \item[(ii)] \'Etale quotient of $\mathbb{P}^1\times C$,
      $C$ a smooth curve with $g(C)\ge 2$, by a
      cyclic group $G$ of automorphisms of $C$ acting freely on $C$ and
      faithfully on $\mathbb{P}^1$.
      If $G$ is not trivial its order is a prime
      $p$ and
      for every $\delta \in ( \mathbb{Z}/p \mathbb{Z})^*$ there exists
      $\varphi \in \mathrm{Aut}(C)$ such that $\varphi \circ g = g^{\pm \delta}
      \circ \varphi$ for every $g\in G$.
    \item[(iii)] $X$ is $\mathbb{P}^1$-bundle over an elliptic curve $E$,
      $\mathbb{P}(\mathcal{O}_E \oplus \mathcal{L})$, where $\mathcal{L}$ is a
      $k$-torsion line bundle on $E$ and either $k=1,2,3$ or
      \begin{itemize}
	\item[(iiia)] $k=4$, $E$ is the elliptic curve relative to the lattice
	  $<1,\frac{1}{2}+\frac{i \sqrt{7}}{2}>$, and $\mathcal{L}$  is in the
	  kernel of either $\frac{3}{2}+i \frac{\sqrt{7}}{2}$ or
	  $\frac{3}{2}-i
	  \frac{\sqrt{7}}{2}$.
	\item[(iiib)] $k=5$, $E$ is the elliptic curve relative to the lattice
	  $<1,i>$, and $\mathcal{L}$ is in the kernel of either $2+i$ or $2-i$.
	\item[(iiic)] $k=7$, $E$ is the elliptic curve relative to the lattice
	  $<1,\frac{1}{2}+\frac{i \sqrt{3}}{2}>$, and $\mathcal{L}$  is in the
	  kernel of either $\frac{5}{2}+\frac{i \sqrt{3}}{2}$ or
	  $\frac{5}{2}-\frac{i \sqrt{3}}{2}$.
      \end{itemize}
  \end{itemize}
\end{theorem}

Looking at the above list, one immediately realizes that a priori it is not
clear whether or not there exist examples of surfaces that satisfies
Theorem \ref{theorem:principal} (ii).
But, it turns out that such examples exist and we describe
some of them in Example
\ref{examples:examplescase2}. Our description is based on the
classical result of Hurwitz that states that every finite group can be realized
as an automorphism group of some compact Riemann surface (see
\cite{breuer:automorphisms_rs} for instance).

The starting point of our analysis is the following result that provides a list
of surfaces that do admit at least one nontrivial self-map.

\begin{theorem} \label{theorem:classification}
  Let $X$ be a complex projective surface, $X$ admits a nontrivial self-map
  if and only if one of the following conditions is satisfied
  \begin{itemize}
    \item[(i)] $X$ is an abelian surface;
    \item[(ii)] $X$ is an hyperelliptic surface, $X$ is an entry
      in the list of Bagnera-de~Franchis (see for example
      \cite[pp.~83--84]{beauville:surfaces});
    \item[(iii)] $X$ is a minimal surface $k(X)=1$ and $\chi(\mathcal{O}_X)=0$;
    \item[(iv)] $X$ is a toric surface;
    \item[(v)] $X$ is a $\mathbb{P}^1$-bundle over an elliptic curve;
    \item[(vi)] $X$ is a $\mathbb{P}^1$-bundle over a nonsingular projective
      curve $B$ with $g(B)>1$ such that $X\times_B B^\prime$ is trivial
      after an \'etale base change $B^\prime \rightarrow B$.
  \end{itemize}
\end{theorem}

\begin{proof}
  See \cite[Theorem~3.2]{fujimoto:endo} for the case $k(X)\ge 0$ and
  \cite[Theorem~3]{nakayama:ruled} for the case $k(X)=-\infty$.
\end{proof}

Since the degree of self-maps is multiplicative with respect to the
composition, one realizes immediately that a surface admits nontrivial self-maps
of any given degree if and only if it admits nontrivial self-maps of any given
degree.
Hence trough the rest of this paper we will restrict our analysis to nontrivial
self-maps of prime degree without any further comment.
Theorem \ref{theorem:principal} will follow from Theorem
\ref{theorem:classification} after a case by case analysis involving both
geometric and arithmetic arguments.

\begin{notations}
  We will denote by $\numeq$, $\lineq$, respectively linear
  equivalence and numerical equivalence of divisors.
  For a locally free sheaf
  $\mathcal{E}$ on a smooth projective variety we put
  $\mathbb{P}(\mathcal{E}):= \mathbf{Proj}\big(
  \mathrm{Sym}(\mathcal{E}^\vee)\big)$, note that our notation
  coincide
  with the $\mathbb{P}(\mathcal{E}^\vee)$ of Hartshorne's book.
\end{notations}

\section{Abelian Surfaces and the case $k(X)\ge 0$}
First of all we are going to analyse case (i) of Theorem
\ref{theorem:classification}, namely abelian surfaces. This case will be a
direct consequence of the Lemma below.

\begin{lemma}
  Let T be a complex torus. There exist an infinite number of primes that do not
  appear as degree of a nontrivial self-map of T.
  \label{lemma:degreetourus}
\end{lemma}

\begin{proof}
  Let $V$ a complex vector space of dimension $g$ and $\Lambda$ a lattice in
  $V$. Put $T=V/\Lambda$.
  Every nontrivial self-map $f:T\rightarrow T$ is the composition of a
  translation and a group endomorphism of $T$, then we may suppose
  without loss of generality that $f$ is a group endomorphism of $T$. Denote by
  $\mathrm{End}(T)$ the set of group endomorphisms of $T$.
  Denote by $\rho_a$ and $\rho_r$ extensions of the analytical and rational
  representation of $ \mathrm{End}(T)$ to $\mathrm{End}_{\mathbb{Q}}(T)=
  \mathrm{End}(T)\otimes \mathbb{Q}$ (see \cite[p.~10]{birkenhake_lange:cav}).
  The extended rational representation
  \begin{equation*}
    \rho_r\otimes 1:\mathrm{End}_{\mathbb{Q}}(T)\otimes \mathbb{C}\rightarrow
    \mathrm{End}_{ \mathbb{C}}(\Lambda \otimes \mathbb{C})\simeq
    \mathrm{End}_{ \mathbb{C}}(V\times V)
  \end{equation*}
  is equivalent to the direct sum of the analytic representation and its
  conjugate \cite[Proposition 1.2.3]{birkenhake_lange:cav}
  \begin{equation}
    \rho_r\otimes 1\simeq \rho_a\oplus \overline{\rho_a}
    \label{equation:torus1}
  \end{equation}
  Observe now that $\deg(f)=\det \rho_r(f)=\det \rho_a(f)
  \overline{\det\rho_a(f)}$. Since $\rho_r(f)$ has integer
  entries its eigenvalues are all algebraic integers, it follows by
  \eqref{equation:torus1} that $\det \rho_a(f)$ is also algebraic integer.
  Moreover $\det \rho_a(f)$ for all $f\in \mathrm{End}(T)$, are
  all contained in the same number field that depends only on $T$. Indeed they
  generate an extension, say $K$,
  contained in a finitely generated extension of $\mathbb{Q}$, namely the
  extension generated by the entries of a period matrix for $T$. It is a well
  known fact that $K$ is finitely generated too
  \cite[p.229, Remark]{lang:algebra}.
  We may suppose without
  loss of generality that $K$ is Galois over $\mathbb{Q}$.
  Summing up we have that if a prime $p$ appears as the degree of an
  endomorphism of $T$ then
  \begin{equation}
    p=\alpha \overline \alpha, \ \alpha\in K
    \label{equation:torus2}
  \end{equation}
  $K$ a Galois number field and $\alpha$ an algebraic integer. 
  Moreover in what follows we may and will restrict
  our attention to primes that do not ramify in the extension $K$, since the
  number of these primes is finite.
  If a prime satisfies \eqref{equation:torus2} then no prime ideal in $K$ that
  divides $(p)$, the ideal generated by $p$ in the ring of integers of $K$,
  admits complex conjugation as its Frobenius. It follows by \v Cebotarev
  density
  theorem \cite[Theorem 13.4, p.~545]{neukirch:ant} that the complementary set
  of the set of primes that satisfy \eqref{equation:torus2} has analytic
  density strictly greater than zero.
\end{proof}

\begin{corollary}
  Let $A$ be an abelian variety, then there are an infinite number of primes
  that do not appear as degree of a nontrivial self-map of $A$.
  \label{corollary:abelianvarieties}
\end{corollary}

\begin{remark}
  Riemann-Hurwitz formula implies that if a curve $C$ possesses
  a nontrivial self-map then $g(C)\le 1$. Moreover by
  Corollary \ref{corollary:abelianvarieties} it follows that $\mathbb{P}^1$ is
  the only curve with nontrivial self-maps of any given degree.
  \label{remark:curves}
\end{remark}

Now we come to (ii) and (iii) of Theorem \ref{theorem:classification}. The
proof of the following Proposition starts considering surfaces of Kodaira
dimension one.
After a preliminary argument we are left with surfaces such that
$p_g=0$ and $q=1$ that we are able to treat by an argument that holds 
for surfaces in Theorem \ref{theorem:classification} (ii) too.

\begin{proposition}
  Let $X$ be either a minimal surface with $k(X)=1$ and $\chi(\mathcal{O}_X)=0$
  or an hyperelliptic surface. $X$ fails to admit a nontrivial
  self-map of degree a given prime for an infinite number of primes.
\end{proposition}

\begin{proof}
  First of all suppose that $X$ is minimal, $k(X)=1$ and
  $\chi(\mathcal{O}_B)=0$. Note that $K_X^2=0$ and then the topological
  Euler-Poincar\'e characteristic $e(X)$ is zero too.
  Moreover $X$ admits an elliptic
  fibration $\pi :X\rightarrow B$, $B$ a smooth curve, and since $e(X)=0$
  exceptional fibres are multiples of a smooth elliptic curve \cite[Proposition
  (11.4) and Remark (11.5), p.~118]{bpv:ccs}.

  Arguing as in
  \cite[Chapter VI]{beauville:endomorphisms} it follows that $X\simeq (F\times
  \widetilde B)/G$ where $F$ and $\widetilde B$ are smooth curves $g(F)=1$,
  $g(\widetilde B)\ge 2$
  and $G$ is a group of automorphisms of $F$ and $\widetilde B$
  such that
  $G$ acts freely on $F\times \widetilde B$.
  Moreover $\widetilde B/G \simeq B$ and $\pi:X\simeq (F\times
  \widetilde B)/G\rightarrow B\simeq \widetilde B/G$ is the map induced by the
  projection of $F\times \widetilde B$ onto $\widetilde B$, in
  particular every smooth fibre of $\pi$ is an elliptic curve isomorphic to
  $F$.

  Observe now that a suitable pluricanonical map factorizes through $\pi$
  and an embedding of $B$ in some projective space
  \cite[Proposition IX.3, p.~108]{beauville:surfaces}.
  Let $f:X\rightarrow X$ a nontrivial self-map.
  Pulling back multiples of canonical
  divisors by $f$ induce a map $f_B:
  B \rightarrow B$ such that the following diagram commutes
  \begin{equation*}
    \xymatrix{ X \ar[r]^{f} \ar[d]_{\pi}& X \ar[d]^{\pi} \\ B \ar[r]^{f_B}
    & B }
  \end{equation*}
  It follows that $\deg f=\deg( f_B)\deg(f_{|F})$ where $\deg(f_{\vert F})$
  denotes the degree of the restriction of $f$ to a smooth fibre.

  If $g(B)> 1$ then $\deg( f_B)=1$ and $X$ fails to admit a nontrivial
  self-map of degree a given prime for an infinite number
  of primes because the same holds for $F$ (Remark \ref{remark:curves}).
  If B is an elliptic curve, any prime appearing as the degree of an
  endomorphism of $X$ is also the degree of an endomorphism of
  the abelian surface $B\times F$, the missing primes are infinite by
  Corollary \ref{corollary:abelianvarieties}.

  If $B$ is rational
  \begin{equation*}
    \mathrm{H}^{1,0}(S) \cong \mathrm{H}^{1,0}(F\times \widetilde B)^G \cong
    \mathrm{H}^{1,0}(F)^G\oplus \mathrm{H}^{1,0}(\widetilde B)^G \cong
    \mathrm{H}^{1,0}(F/G)\oplus \mathrm{H}^{1,0}(B).
  \end{equation*}
  Since $\chi(\mathcal{O}_X)=0$ we have $q=1$ and $F/G$ is elliptic.
  The fibres of the natural
  map $\alpha: (F\times \widetilde B)\rightarrow F/G$ are connected and
  isomorphic to $\widetilde B$, then
  $\alpha$ is the Albanese map of $X$. By the universal property of the
  Albanese
  map there exists a $\varphi: F/G \rightarrow F/G$ such that $\alpha \circ f=
  \varphi \circ \alpha$. Again $\deg(f)=\deg(\varphi)\deg(f_ B)$ and we
  can conclude as above.

  The case of hyperlliptic surfaces is analogous, since in this case the
  Albanese variety is an elliptic curve and the fibres of the  Albanese map
  are isomorphic elliptic curves.
\end{proof}

\section{The case $k(X)=-\infty$: toric surfaces}
We are going to show that, in the case of a toric surfaces,
the presence of curves of
negative self-intersection implies that,
up to a finite set,
degrees of nontrivial self-maps
of the surface are not square free. But
before dealing with the specific case of toric surfaces, we briefly recall how
a general nontrivial self-map of a surface $X$ acts on the set of curves
of $X$.

\begin{remark} \label{remark:negativecurves}
  Let $X$ be a surface with a nontrivial surjective endomorphism,
  $f:X\rightarrow X$.
  Let $C$ and $D$ be
  irreducible curves on $X$ such that $f(C)=D$.
  Since $f_*\circ f^*=\deg(f) \mathrm{Id}$ on $\mathrm{NS}(X)$,
  there
  are positive integers $a,b$ such that $f_*( C)\numeq a D$ and $f^*(
  D)\numeq b C$ where
  $\deg f=ab$.
  As a consequence, $C^2=0$ if and only if $D^2=0$,
  hence the image under a nontrivial self-map of an irreducible curve
  with zero self-intersection 
  is always a curve with the same property.
  Analogously $ C^2<0$ if and only if $ D^2<0$ and in this case we
  also have $f^{-1}(D)=C$: in fact any two distinct components $f^{-1}(D)$
  should be curves with negative self-intersection whose classes in
  $\mathrm{NS}(X)$ are linearly dependent
  (by injectivity of $f_{*}$).
  Hence, if we denote
  \begin{equation*}
    \mathcal{S}_X=\{C\ \mathrm{irreducible}\ \mathrm{curve}|\ C^2<0\}
  \end{equation*}
  the map of sets $\hat{f}:
  \mathcal{S}_X\rightarrow \mathcal{S}_X$, $C\mapsto f(C)$ is bijective.
\end{remark}

In the proof of Proposition \ref{proposition:toric} below
we will make use of the
following elementary statement. Since we will need it in the subsequent
sections
we state it in the form of a Lemma.

\begin{lemma}
  Let $X$ be a surface that contains one and only one curve $C$ such that
  $C^2<0$. If $f:X\rightarrow X$ is a nontrivial self-map then $\deg(f)$ is a
  square.
  \label{lemma:negativecurve}
\end{lemma}

\begin{proof}
  In view of the above Remark \ref{remark:negativecurves}
  we have $f(C)=C$, $f^* C\numeq a_1 C$ and $f_* C\numeq a_2 C$
  for suitable integers such
  that $\deg f=a_1a_2$. By the Projection Formula
  \begin{equation*}
    a_1 C^2=f^* C \cdot C = C \cdot f_* C =a_2
    C \cdot  C = a_2  C^2
  \end{equation*}
  hence $a_1=a_2$ and $\deg f$ is a square.
\end{proof}

\begin{proposition}
  The only toric surface admitting nontrivial self-maps
  of any given degree is $\mathbb{P}^1\times \mathbb{P}^1$.  If $S \not\simeq
  \mathbb{P}^1\times \mathbb{P}^1$ is a toric surface, $S$ fails to have a
  nontrivial self-map of prime degree for all but at most a
  finite number of primes. \label{proposition:toric}
\end{proposition}

\begin{proof}
  Let $X$ be a toric surface, by the classification of toric surfaces,
  see \cite[Theorem 1.28, p.~42]{oda:convex} for example, $X$ is obtained
  by a finite number of equivariant blow-ups from the projective plane or
  a Hizerbruch surface $\mathbb{F}_n$, $n \ge 0$.  The only cases in which
  $\mathcal{S}_X$ is empty are either $X\simeq \mathbb{P}^2$ or $X\simeq
  \mathbb{P}^1\times \mathbb{P}^1$.
  All the nontrivial self-maps of the former have degree a square,
  the latter instead has nontrivial self-maps of
  any given degree.

  Suppose now that $\mathcal{S}_X$ is nonempty. Since $X$ is a toric
  surface, any
  irreducible curve on $X$ with negative self-intersection is included in
  the complement of the torus,
  hence $\mathcal{S}_X$ is finite.
  If $\mathcal{S}_X$ consists of
  only one element, then our claim follows by Lemma
  \ref{lemma:negativecurve}.
  If $\mathcal{S}_X$
  contains more than one element then $X$ is not minimal as it follows by the
  classification of rational surfaces.
  Let $D\in
  \mathcal{S}_X$ be a $-1$-curve on $X$ and $C\in \mathcal{S}_X$ such
  that $f(C)=D$, recall Remark \ref{remark:negativecurves}. We have
  \begin{equation} \label{equation:toric}
    b^2  C^2 =f^*D \cdot f^*D =f^*( D \cdot D)= -\deg(f)
  \end{equation}
  and since $\mathcal{S}_X$ is finite $C^2$ in \eqref{equation:toric}
  can only take a finite number of values. It follows that apart from a finite
  number of values $\deg(f)$ is not square free.
\end{proof}

\section{The case $K(X)=-\infty$: $\mathbb{P}^1$-bundles over a
non rational curve}
First
of all we are going to introduce notations that will be used in this and
the next Section.

\begin{notations}
  We denote by $\mathcal{E}$ a rank two vector bundle over a curve $B$
  of genus $g(B)$ greater than or equal to one, and put
  $X=\mathbb{P}(\mathcal{E})$. Moreover we denote by $\pi:X\rightarrow B$ the
  projection associated to the projective bundle structure. Since $g(B)\ge
  1$, given a nontrivial self-map $f:X\rightarrow X$,
  it induces a nontrivial self-map of $B$ that we denote by $f_{B}$.
\end{notations}

Note that $\deg(f)=\deg(f_B)\cdot \deg(f_{\mathbb{P}^1})$, where
$\deg(f_{\mathbb{P}^1})$ denotes the degree of $f$ when restricted to a
fibre, this degree does not depend on the particular chosen fibre. Moreover
if $\deg(f_B)>1$ i.e. $f_B$ is a nontrivial surjective endomorphism, $B$
is an elliptic curve.

We are going to analyze case (vi) of Theorem \ref{theorem:classification}.
The main result of this section is the following:

\begin{theorem}
  Let $\mathcal{E}$ be a rank two vector bundle over a projective curve $B$
  of genus $g>1$.
  The projective bundle $X=\mathbb{P}(\mathcal{E})$ admits 
  a nontrivial self-map 
  of degree $n$ for every positive integer $n$ if and only if
  \begin{enumerate}
    \item[(i)] $\mathcal{E}\simeq
      \mathcal{O}\oplus \mathcal{L}$ where $\mathcal{L}$ is either a trivial or
      a torsion line bundle of order $p$ a prime and in the latter case
    \item[(ii)] for every $m\in (\mathbb{Z}/p \mathbb{Z})^*$ there exists an
      automorphism $\varphi\in \mathrm{Aut}(B)$
      such that $\varphi^{*}(\mathcal{L})\simeq \mathcal{L}^{\pm m}$.
  \end{enumerate}
  \label{theorem:projectivebundles}
\end{theorem}

Before proving the above Theorem we recall some well known facts on
$\mathbb{P}^{1}$-bundles over a smooth projective curve, and in the meantime we
will establish some notations.

\begin{remark}
  \label{remark:nsproj}
  Suppose  that $X=\mathbb{P}(\mathcal{E})$ where $\mathcal{E}$
  is a locally free sheaf of rank
  two vector bundle of degree $e$ on the smooth curve $B$.
  We have $\mathrm{NS}(X)=\mathbb{Z}
  H+\mathbb{Z} F$, were $H$ is a divisor such that
  $\mathcal{O}_X(H)\simeq \mathcal{O}(1)$
  and $F$ is a fibre of the
  projection $\pi$. If $D$ is a divisor on $X$ such that $D^2=0$ its class
  in $\mathrm{NS}(X)$ is either a multiple of $F$ or a multiple
  of $H+\frac{e}{2}F$.

  If $f:X\rightarrow X$ is a nontrivial self-map
  that induces an automorphism on the base $B$
  the
  ramification divisor $R$ of $f$ satisfies
  $R\lineq K_{X/B}-f^*K_{X/B}$ and
  moreover
  we have $K_{X/B}\numeq -2H-eF$.
  It follows that $K_{X/B}^2=0$ and $(f^*K_{X/B})^2=0$ which
  implies $f^*K_{X/B}\numeq -2(\deg f)H-(\deg f)eF$. Summing up
  we have
  \begin{equation*}
    R\numeq (1-\deg f)K_{X/B}\ \mathrm{and}\ R^2=0\ .
  \end{equation*}
\end{remark}

\begin{remark}
  \label{remark:normals}
  If $X=\mathbb{P}(\mathcal{E})$ admits two disjoint sections,
  $\mathcal{E}$ is isomorphic to
  the direct sum of two line bundles, and then
  $X=\mathbb{P}(\mathcal{O}\oplus \mathcal{L})$
  up to an isomorphism obtained tensoring by a suitable line bundle.
  If $\mathcal{L}$ is nontrivial, we denote by $S_1,S_2$ the
  sections of $X$ corresponding respectively to
  the line bundles
  $\mathcal{O}$ and $\mathcal{L}$, and by $s_1,s_2:B\rightarrow
  X$ the associated embeddings.
  We have $\mathcal{N}_{S_1/X}\simeq \mathcal{N}_{S_2/X}^\vee \simeq
  \mathcal{L}$,
  $S_1^2= \deg (\mathcal{L})=e$ and $S_2^2= -\deg(\mathcal{L})=-e$.

  If $\deg(\mathcal{L})>0$, $S_2$ is the only curve on $X$ of negative
  self-intersection. In case $\deg(\mathcal{L})=0$ and $\mathcal{L}$ is not
  trivial, the curves $S_1,S_2$ are the only sections of zero
  self-intersection
  and any class in $\mathrm{NS}(X)$ whose square equals zero is
  either a multiple of the class of $H$ or a multiple of the class of $F$.
\end{remark}

\begin{remark}\label{remark:torsion}
  Curves dominating $B$ and
  of zero self-intersection on $X= \mathbb{P}(\mathcal{O}\oplus
  \mathcal{L})$, $\mathcal{L}$ a torsion line bundle of order $k$, will play a
  central role in the proof of Theorem \ref{theorem:projectivebundles}.
  A particular	curve of this type is given by the \'etale cyclic cover
  $j:\widetilde B\rightarrow B$
  of degree $k$ determined by $\mathcal{L}$.
  We are going to recall how to construct $\widetilde B$, see for instance
  \cite[p.~54]{bpv:ccs}. Denote by $\mathbf{L}$ the total space of
  $\mathcal{L}$, and by $\mathrm{pr}_{\mathbf{L}}: \mathbf{L}\rightarrow B$ the
  bundle projection. The zero divisor of the section
  $1-l^k$ in $\mathbf{L}$, where $l\in
  \Gamma(\mathbf{L},\mathrm{pr}_{\mathbf{L}}^*\mathcal{L})$ is the tautological
  section, is the curve $\widetilde B$ and $j$ is the restriction to
  $\widetilde
  B$ of the bundle map.
  Observe that there is a canonical isomorphism of varieties
  $\mathbf{L}\simeq X\setminus S_{2}$, and that through this isomorphism the
  image of $\widetilde B$ is disjoint from $S_1$.
  By the above description
  the projective irreducible curve $\widetilde B$
  is a principal divisor in
  $X\setminus S_{2}$ and the normal bundle to $\widetilde B$ in $X$
  is trivial.

  Let $D\ne S_{1}, S_{2}$ be an irreducible curve in $X$ dominating $B$
  such that $D^2=0$, we are going to show that
  $D\subset X\setminus S_{2}\simeq \mathbf{L}$ and there exists
  $a\in \mathbb{C}^*$
  such that the automorphism of varieties
  \begin{equation*}
    \mu_a: \mathbf{L} \rightarrow \mathbf{L}
  \end{equation*}
  induced by multiplication by $a$,
  sends $D$ to $\widetilde B$.
  Since $D^2=0$ and $D$ dominates $B$,
  the divisor $D$ is numerically
  equivalent to a multiple of $S_{1}$ and $S_{2}$,
  in particular it is disjoint from $S_{1}\cup S_{2}$ and
  thus
  lies in
  the complement of the zero section of $X\setminus S_{2}\simeq
  \mathbf{L}$ that coincides with $S_1$.
  Since $\widetilde B$ is a principal divisor in $X\setminus S_2$,
  the intersection of $\widetilde B$
  with a different projective curve included in $X \setminus S_2$
  is empty.
  Since multiplication by scalars acts transitively on the non zero elements of
  the fibres of $\mathbf{L}$,
  there exists $a\in \mathbb{C}^{*}$ such that
  $\widetilde B\cap \mu_a(D)$ is non empty and this forces $\widetilde
  B=\mu_a(D)$.
  In particular the restriction of $\pi$ to such a $D$ gives an \'etale
  covering of $B$ isomorphic to $j:\widetilde B \rightarrow B$, and $D\cdot
  F=k$.
\end{remark}

Theorem \ref{theorem:projectivebundles} will be a consequence of the next
two Propositions.
In Proposition \ref{proposition:projciclic}
we study nontrivial self-maps of surfaces of the form
$\mathbb{P}(\mathcal{O}\oplus \mathcal{L})$ where $\mathcal{L}$
is a torsion line bundle on a curve $B$ of genus $g(B)\ge1$.
In Proposition \ref{proposition:cases} we characterize $\mathbb{P}^1$-bundles
on a curve of genus greater then one admitting an endomorphism of degree two.

\begin{proposition}
  \label{proposition:projciclic}
  Suppose $X=\mathbb{P}(\mathcal{O}\oplus \mathcal{L})$,
  where $ \mathcal{L}$ is a torsion line bundle on $B$ of order $k>1$, and the
  genus of $g(B)\ge 1$.
  $X$ admits a nontrivial self-map $f:X\rightarrow X$
  such that $f_B \in \mathrm{Aut}(B)$ and $\deg (f)=d$
  if and
  only if
  either
  \begin{enumerate}
    \item [(i)] $k|d$, or
    \item [(ii)] there exists $\varphi\in \mathrm{Aut}(B)$ such that
      $\varphi^*\mathcal{L}\simeq \mathcal{L}^m$ with
      $d\equiv \pm m \
      \mathrm{mod}\ k$.
  \end{enumerate}
  In case (ii) there exists a nontrivial self-map $f$ of degree $d$
  such that $f_B=\varphi$ and
  either $f^*S_i=dS_i$ for $i=1,2$ or $f^*S_i=dS_j$ for $i\ne j$.
\end{proposition}

\begin{proof}
  We begin by proving the only if part of our statement.
  Since $f_B$ is an
  automorphism $f(S_i)$ is a section. Moreover, see Remark
  \ref{remark:negativecurves}, $(f_*(S_i))^2=0$,
  therefore either $f(S_{i})=S_{1}$ or
  $f(S_{i})=S_{2}$ for $i=1,2$.
  Then there are three possible cases
  \begin{gather}
    f(S_{1})=S_{1}\ \mathrm{and}\ f(S_{2})=S_{2} \label{equation:case1} \\
    f(S_{1})=S_{2}\ \mathrm{and}\  f(S_{2})=S_{1} \label{equation:case2} \\
    f(S_{1})=f(S_{2})=S_{i}\ \mathrm{for}\ \mathrm{either}\ i=1\ \mathrm{or}\
    i=2
    \label{equation:case3}
  \end{gather}
  and in any of the cases above the pullback of the divisor $S_{j}$ is given by
  \begin{equation} \label{equation:pullback}
    f^*S_j=n_{1,j}S_1+n_{2,j}S_2+\sum_{\iota}k_\iota C_\iota
  \end{equation}
  where $n_{1,j},n_{2,j}$ are nonnegative integer,
  $k_\iota$ are positive integers,
  and $C_\iota\ne
  S_1,S_2$ are distinct irreducible curves.
  Since $X$ contains no curve with negative self-intersection, the irreducible
  components of
  $f^*(S_j)$ are disjoint, $C_\iota$ dominates $B$,
  and $C_\iota ^{2}=0$ for every $\iota$.
  By Remark
  \ref{remark:torsion} we also know that  $C_{\iota}$ intersects
  transversally every fibres
  of $\pi$ in $k$ points.
  Since $f^{*}(F)\lineq F$ we are able to recover the degree of $f$
  as the intersection number between $F$ and $f^*S_j$ hence
  $$d=n_{1,j}+n_{2,j}+k\sum_{\iota}k_\iota .$$
  In case \eqref{equation:case3} there exists $j$ such that
  $n_{1,j}=n_{2,j}=0$ and the
  intersection number between $f^*S_j$ and $F$ is  $k\sum_{\iota}k_\iota$.
  In cases \eqref{equation:case1}--\eqref{equation:case2} we have
  $f(S_{1})=S_{j}$ the multiplicity $n_{2,j}$ is zero and
  $d\equiv n_{1,j}\ \mathrm{mod}\ k$, $j=1,2$. We have then
  \begin{equation}
    f^*_B \mathcal{N}_{S_j/X}\simeq f_B^* s_j^* \mathcal{O}(S_j)\simeq s_1^*f^*
    \mathcal{O}(S_j)\simeq s_1^* \mathcal{O}(n_{1,j}S_1)\simeq
    (\mathcal{N}_{S_1/X})^{
    n_{1,j}}\ .
    \label{equation:normal}
  \end{equation}
  Since $\mathcal{N}_{S_1/X}\simeq \mathcal{N}_{S_2/X}^{\vee}\simeq
  \mathcal{L}$
  we get
  $f^*_B \mathcal{L}\simeq \mathcal{L}^{ n_{1,1}}$ in case
  \eqref{equation:case1}
  and $f^*_B
  \mathcal{L}\simeq \mathcal{L}^{ -n_{1,2}}$ in case \eqref{equation:case2}.
  Setting $\varphi=f_{B}$ we get
  $\varphi^*\mathcal{L}\simeq \mathcal{L}^m$ with
  $d\equiv	m \
  \mathrm{mod}\ k$ for $m=n_{1,1}$ in case \eqref{equation:case1} and
  $\varphi^*\mathcal{L}\simeq \mathcal{L}^m$ with
  $d\equiv	- m \
  \mathrm{mod}\ k$ for $m=-n_{1,2}$ in case \eqref{equation:case2}.

  Now we come to the proof of the if part of our statement.
  Denote by $\mathbf{O}, \mathbf{L}$
  respectively
  the total spaces of $\mathcal{O}$ and
  $\mathcal{L}$ and by $\mathbf{O}_x$ and $\mathbf{L}_x$ fibres of the bundle
  map over the point $x$.
  First of
  all suppose that $k|d$.
  In this case,  the  isomorphism $i:\mathbf{L}^{\otimes d}
  \rightarrow \mathbf{O}$
  enables us to construct a degree $d$ morphism of varieties
  \begin{gather*}
    F: \mathbf{O}\oplus \mathbf{L} \rightarrow \mathbf{O}\oplus
    \mathbf{L}^{\otimes d-1} \\
    (\alpha,l)\in \mathbf{O}_{x}\oplus \mathbf{L}_{x}\mapsto
    (\alpha^{d}+i(l^{\otimes d}), \alpha\cdot l^{\otimes d-1})\in
    \mathbf{O}_{x}\oplus \mathbf{L}^{\otimes d-1}_{x}.
  \end{gather*}
  Since $F$ is homogeneous on the fibres it induces a degree $d$ morphism
  $\hat{F}:\mathbb{P}(\mathcal{O}\oplus \mathcal{L}) \rightarrow
  \mathbb{P}(\mathcal{O}\oplus \mathcal{L}^{d-1})$ and composing
  with the canonical isomorphism $\mathbb{P}(\mathcal{O}\oplus
  \mathcal{L}^{d-1})\simeq \mathbb{P}(\mathcal{L}\oplus \mathcal{O})$
  induced by tensorization by $\mathcal{L}$ we get the desired endomorphism.

  Suppose now that $d\equiv \pm m\ \mathrm{mod}\ k$ and that
  there is an automorphism of $B$, $\varphi$ such that
  $\varphi^*\mathcal{L}\simeq
  \mathcal{L}^m$.
  Since $\mathcal{L}$ is of $k$-torsion, we have a degree $d$ map
  \begin{equation*}
    \mathbf{O} \oplus \mathbf{L} \rightarrow \mathbf{O}
    \oplus \mathbf{L}^{d} \simeq \mathbf{O}\oplus \mathbf{L}^{\pm m}
    \simeq \mathbf{O}\oplus \varphi^*\mathbf{L}^{\pm 1}
  \end{equation*}
  given on the fibres by
  \begin{equation*}
    (\alpha,l)\mapsto  (\alpha^{d},  l^{\otimes d})
  \end{equation*}
  and hence a map $\phi_1:\mathbb{P}(\mathcal{O}\oplus
  \mathcal{L})\rightarrow
  \mathbb{P}(\mathcal{O}\oplus \varphi^*\mathcal{L})$ of degree $d$ that
  induces the identity on the base. Moreover the natural map
  \begin{equation*}
    \phi_2: \mathbb{P}(\mathcal{O}\oplus
    \varphi^*\mathcal{L}^{\pm 1} )\simeq \mathbb{P}(\mathcal{O}\oplus
    \mathcal{L}^{\pm 1})\times_\varphi B\rightarrow
    \mathbb{P}(\mathcal{O}\oplus
    \mathcal{L}^{\pm 1})\simeq \mathbb{P}(\mathcal{O}\oplus
    \mathcal{L})
  \end{equation*}
  is an
  isomorphism that induces $\varphi$ on the base. It follows that
  $\phi=\phi_1\circ \phi_2$ is a nontrivial self-map of
  $\mathbb{P}(\mathcal{O}\oplus
  \mathcal{L})$ that induces $\varphi$ on the base.
  The final part of the statement holds setting $f=\phi$.
\end{proof}

\begin{proposition}
  Let $X=\mathbb{P}(\mathcal{E})\rightarrow B$ be a projective bundle,
  with $\mathcal{E}$ a locally free sheaf of rank two on $B$ a curve of genus
  $g(B)\ge 2$. Suppose $X$ admits a nontrivial self-map of
  degree two, then either
  \begin{itemize}
    \item[(i)] $X=\mathbb{P}(\mathcal{O}\oplus \mathcal{L})$ with
    $\mathcal{L}$ a
      torsion line bundle, or
    \item[(ii)] The ramification divisor $R_{f}$ of $f$ is a smooth irreducible
      curve, the restriction of $\pi$ to $R_{f}$ is  an \'etale double covering
      of $B$ and the normal bundle $\mathcal{N}_{R_{f}/X}$ to $R_{f}$ in $X$
      is a torsion line bundle of order strictly greater than $2$.
  \end{itemize}
  \label{proposition:cases}
\end{proposition}

\begin{proof}
  Let $f:X\rightarrow X$ be a nontrivial self-map of degree two. Since $f_B$
  is an automorphism, the restriction of $f$ to every fibre is a double
  covering of $\mathbb{P}^1$, hence it ramifies at exactly two points.
  It follows that $R_{f}$ is a smooth curve intersecting transversally every
  fibre of $\pi$ in two points.
  Therefore $R_f$ is either union of two disjoint sections $S_{1}$ and $S_{2}$
  or it is an \'etale double covering of $B$.

  In the first case  each one of this sections has
  zero self-intersection by Lemma \ref{lemma:negativecurve}.
  If $X$ is not the trivial projective bundle $S_{1}$ and $S_{2}$ are the
  unique sections zero self-intersection. Moreover the image
  $T=f(S_{1})$ is a section of zero self-intersection by Remark
  \ref{remark:negativecurves}, therefore $T=S_j$ for either $j=1$ or $j=2$.
  As a particular case of equation \eqref{equation:nontriext} we get
  \begin{equation*}
    f^*_B \mathcal{N}_{S_j/X}\simeq (\mathcal{N}_{S_1/X})^2
  \end{equation*}
  and since $\mathcal{N}_{S_1/X}\simeq \mathcal{N}_{S_2/X}^{\vee}$ we obtain
  \begin{equation*}
    \mathrm{either}\ f^*_B (\mathcal{N}_{S_1/X}^{\vee})\simeq
    (\mathcal{N}_{S_1/X})^2
    \quad \mathrm{or}\quad
    f^*_B \mathcal{N}_{S_1/X}\simeq (\mathcal{N}_{S_1/X})^2.
  \end{equation*}
  In both cases, since $f_B$ has finite order in the group of automorphisms of
  $B$,
  $\mathcal{N}_{S_1/X}$
  (hence also $\mathcal{N}_{S_2/X}$) is a torsion line bundle.

  We turn our attention to case (ii), namely when $R_f$ is irreducible.
  We denote by $i:X\rightarrow
  X$ the involution associated with $f$.
  We may assume that
  $X$ contains no sections with zero self-intersection.
  Otherwise, denote by $S$ such a section, the curves
  $S$ and $S':=i(S)$ are numerically equivalent, hence they are disjoint. It
  follows that
  $\mathcal{N}_{S/X}=\mathcal{N}_{S'/X}^{\vee}$.
  On the other hand, since $i$ descends to the identity on $B$ it induces
  an isomorphism $\mathcal{N}_{S/X}\simeq\mathcal{N}_{S'/X}$.
  We conclude $X\simeq \mathbb{P}(\mathcal{O}\oplus \mathcal{N}_{S/X})$
  and $\mathcal{N}_{S/X}^{\otimes 2}\simeq \mathcal{O}$ and we are in case (i).

  We may also assume that $f(R_{f})=R_{f}$. Otherwise the image $T=f(R_{f})$
  satisfies, by Remark \ref{remark:negativecurves}, $T^{2}=0$ and since
  by the above argument we may suppose that $f(T)$
  is not a section, we have that $T\ne i(T)$.
  Since $T$ and $i(T)$ are numerically equivalent they are disjoint too.
  Hence making a base change by the \'etale double covering
  $\pi_{|T}:T\rightarrow B$
  we get a  $\mathbb{P}^{1}$-bundle over $T$ with four disjoint sections
  (two of them mapping onto $T$ and the others onto $i(T)$).
  So this projective bundle is trivial and $X$ is the quotient of $T\times
  \mathbb{P}^{1}$ by a $\mathbb{Z}/2\mathbb{Z}$-action without
  fixed points. Such an action is always diagonal (see Remark
  \ref{remark:action}),
  hence there exist $p_{1}, p_{2}\in \mathbb{P}^{1}$
  such that $T\times p_{i}$ is sent to itself  by the
  $\mathbb{Z}/2\mathbb{Z}$-action.
  The image $S_{i}$ of $T\times p_{i}$ in $X$ is a section and the
  pullback of its normal bundle to
  $T\times p_{i}$ is trivial. Hence $\mathcal{N}_{S_{i}/X}$ is a torsion
  line bundle of order two and we are again in case (i) of the Proposition.

  Finally assuming $f(R_f)=R_f$ and denoting by $\hat{f}:R_{f}\rightarrow
  R_{f}$ the restriction of $f$ we get
  $\hat{f}^{*}\mathcal{N}_{R_{f}/S}\simeq \mathcal{N}_{R_{f}/S}^{2}$ and since
  $\hat{f}$ has finite order, the normal bundle
  $\mathcal{N}_{R_{f}/S}$ is a torsion line bundle. We can exclude that
  $\mathcal{N}_{R_{f}/S}^{2}\simeq \mathcal{O}$. Indeed this would imply that
  $\mathcal{N}_{R_{f}/S}=\mathcal{O}$, and after a base
  change by the restriction of
  $\pi$ to $R_{f}$ we get a $\mathbb{P}^{1}$-bundle over $R_{f}$ having two
  sections with trivial normal bundles, so again it is
  the trivial $\mathbb{P}^{1}-$bundle and $X$ is as in case (i).
\end{proof}

We are now in position to prove Theorem \ref{theorem:projectivebundles}.

\begin{proof}[Proof of Theorem \ref{theorem:projectivebundles}].
  Let $X=\mathbb{P}(\mathcal{E})$ be a surface admitting a nontrivial self-map
  of degree $n$ for every $n\in \mathbb{N}$.
  By Proposition \ref{proposition:cases} either
  \begin{itemize}
    \item[(i)] $X=\mathbb{P}(\mathcal{O}\oplus \mathcal{L})$
      with $\mathcal{L}$ a torsion line bundle, or
    \item[(ii)] there exists a nontrivial self-map
      of $X$ of degree two
      such that the restriction of $\pi$ to the ramification divisor
      $R_{f}$ of $f$ is a nontrivial \'etale double covering of $B$
      and the normal bundle $\mathcal{N}_{R_{f}/X}$ to $R_{f}$ in $X$
      is a torsion line bundle of order strictly greater than $2$.
  \end{itemize}
  In case (i) we may suppose that $\mathcal{L}$ is not trivial since otherwise
  our statement is clearly true.
  Let $k\ge 2$ be the order of $\mathcal{L}$. By Proposition
  \ref{proposition:projciclic}, for every
  non zero $r \in \mathbb{Z}/k\mathbb{Z}$
  there exists $\varphi \in \mathrm{Aut}(B)$ such that
  $\varphi^{*}(\mathcal{L})\simeq
  \mathcal{L}^{\pm r}$.
  We only need to remark that $k$ must be prime since
  pulling back by an
  automorphism of $B$ preserves
  the order of a torsion line bundle.

  In case (ii) let $m$ be the order of $\mathcal{N}_{R_{f}/X}$.
  First of all
  we claim that for any curve $C\ne R_{f}$ on $X$ such that $C^{2}=0$, the
  intersection number $C\cdot F$ is a multiple of $m$.
  In fact the \'etale double covering $R_{f}\rightarrow B$ induces an \'etale
  double covering $h:R_f \times_{B} X\rightarrow X$ and $h^{-1}(R_{f})$
  is the union of two
  disjoint sections $S_1$ and $S_2$ whose normal bundles are
  $\mathcal{N}_{R_{f}/X}$ and $\mathcal{N}_{R_{f}/X}^{\vee}$.
  Hence  $R_f \times_{B} X\simeq \mathbb{P}(\mathcal{O}\oplus
  \mathcal{N}_{R_{f}/X})$.
  Let $C'$ be a component of $h^{-1}(C)$, since $\mathbb{P}(\mathcal{O}\oplus
  \mathcal{N}_{R_{f}/X})$
  contains no curve with
  negative self-intersection we have $C'^{2}=0$.
  Moreover $C'\ne S_{1},S_{2}$ because $C\ne R_{f}$. By Remark
  \ref{remark:torsion}
  $C'$ is a degree $m$ cover of $R_{f}$ and, since the degree of $h$ is two,
  the degree of the restriction of $\pi$ to $C=h(C')\subset X$ is either $m$
  or $2m$. In both cases $m$ divides $C\cdot F$.

  We are going to show now that $X$ does not admit a nontrivial self-map of
  degree
  $m$. Indeed,
  let $g:X\rightarrow X$ be such a map.
  Pulling back the divisor
  $R_{f}$ we get
  \begin{equation}
    \label{equation:epullback}
    g^{*}(R_{f})=aR_{f}+\sum b_{i}C_{i}
  \end{equation}
  where $b_{i}>0$ and $C_{i}^{2}=0$ for all $i$.
  By Remark \ref{remark:negativecurves} the curve $g(R_f)$ has
  self-intersection
  zero, moreover since $g$ induces an automorphism on the base $B$, we have
  $g_*(R_f)\cdot F=2<m$. We deduce, by the above
  claim on curves of self-intersection
  zero that
  $g(R_{f})=R_{f}$. Therefore $a$ is strictly positive too.
  Intersecting both members of equation \eqref{equation:epullback}
  with $F$ we get
  \begin{equation*}
    2m=2a+bm\ .
  \end{equation*}
  Since $a>0$ either $b=1$ and $2a=m$ or $b=0$ and $a=m$. In both cases
  $a\ge 2$
  because $m>2$.
  Let $\hat{g}:R_{f}\rightarrow R_{f}$ be the restriction of $g$ and denote
  by $i:R_{f}\rightarrow X$ the closed embedding. Then
  $$\hat{g}^{*}(\mathcal{N}_{R_{f}/S}))=i^{*}
  g^{*}(\mathcal{O}(R_{f}))=i^{*}(\mathcal{O}(aR_{f}))=
  \mathcal{N}_{R_{f}/S}^{a}\ .$$
  This is absurd because $\hat{g}$ is an automorphism, $\mathcal{N}_{R_{f}/S}$
  is a torsion line bundle of order $m$ and
  $\mathcal{N}_{R_{f}/S}^{a}$ is a torsion line bundle of order $m/a<m$.
  This proves the 'only if' part of the statement, the 'if' part follows
  directly from Proposition \ref{proposition:projciclic}.
\end{proof}

It is not immediately clear whether or not surfaces satisfying the
characterization of Theorem \ref{theorem:projectivebundles} do exist.
In order to provide examples of such surfaces, we are going to reformulate
Theorem \ref{theorem:projectivebundles} in terms of \'etale quotients of
trivial
projective bundles, this is the content of Corollary \ref{corollary:togroups}.
This will lead us to explicit examples, see
Examples \ref{examples:examplescase2}.

We start with the following Remark in witch among other things we will fix
notations needed later on.

\begin{remark}
  Let $C$ be a smooth projective curve of genus
  $g(B)>1$.
  Let $h:\mathbb{P}^{1}\times C \rightarrow
  \mathbb{P}^{1}\times C$ a nontrivial self-map,
  since
  $g(C)>1$ we have
  $h=h_{1}\times h_{2}$ where $h_{1}:\mathbb{P}^{1}\rightarrow \mathbb{P}^{1}$
  is a nontrivial self-map and
  $h_{2}: C\rightarrow C$
  is an automorphism.

  Let $G$ be a cyclic group  of automorphisms of $C$
  that acts on $\mathbb{P}^{1}\times C$.
  For every
  $g\in G$ we denote by $\Phi_{g}:\mathbb{P}^{1}\times C\rightarrow
  \mathbb{P}^{1}\times C$ the induced automorphism. In particular
  $\Phi_{g}=\phi_g \times g$ where $\phi_g$ is an
  automorphisms of
  $\mathbb{P}^{1}$, it follows that $G$
  acts on $\mathbb{P}^{1}$ too and the action on
  $\mathbb{P}^{1}\times C$ is the induced
  diagonal action.
  Moreover if the order of $G$ is $m$, there exists an affine coordinate $z$
  on $\mathbb{P}^{1}$ such that $\phi_g(z)=\epsilon z$ where $\epsilon$
  is a $m$-th root of the unity.
  \label{remark:action}
\end{remark}

\begin{corollary}
  Let $X$ be a nontrivial $\mathbb{P}^1$-bundle over a smooth projective
  curve $B$ of genus
  $g(B)>1$. The surface $X$ admits nontrivial self-maps
  of any given degree if and only if there exist a prime
  number $p$, a curve $C$ and $G$ a group of automorphisms of
  $C$ acting on $\mathbb{P}^1$ such that
  \begin{itemize}
    \item[(i)] $G$ is cyclic of order $p$ ,
      it acts freely on $C$, faithfully on $\mathbb{P}^1$ and
      $X$ is a Galois
      \'etale quotient $X\simeq (\mathbb{P}^1\times C)/G$.
    \item[(ii)] For every non zero $\delta \in \mathbb{Z}/p\mathbb{Z}$ there
      exists an
      automorphism $\varphi \in \mathrm{Aut}(C)$
      such that for every $g\in G$ either
      $\varphi \circ g=g^{\delta}\circ \varphi$
      or
      $\varphi \circ g= g^{-\delta}\circ \varphi$.
  \end{itemize}
  \label{corollary:togroups}
\end{corollary}

\begin{proof}
  Let $X$ be a nontrivial $\mathbb{P}^1$-bundle over a curve $B$
  admitting nontrivial self-maps of any given degree.
  We are going to prove
  that there exists a prime $p$ such that (i) and (ii) hold.
  By Theorem \ref{theorem:classification}, $X\simeq
  \mathbb{P}(\mathcal{O}\oplus \mathcal{L})$ where $\mathcal{L}$ is a
  $p$-torsion line bundle on $B$, $p$ a suitable prime.
  Let $j:C \rightarrow B$ be the
  Galois \'etale cover of $B$ determined by $\mathcal{L}$, $G\subset
  \mathrm{Aut}(C)$
  its Galois group, $G$ is cyclic of order $p$.
  Let $\overline{j}: X\times_B C \rightarrow X$ be the
  induced Galois \'etale cover. Since $\mathcal{L}$ pulls
  back to a trivial line bundle on $C$,
  \begin{equation*}
    X\times_B C \simeq \mathbb{P}^1\times C
  \end{equation*}
  and $X$ is the quotient of $X\times_B C$ by a
  $G$-action.
  The action of $G$ on $\mathbb{P}^1\times C$ is diagonal (see Remark
  \ref{remark:action}),
  moreover $G$ acts freely on $C$,
  and faithfully on $\mathbb{P}^1$, otherwise $X$ would be isomorphic to
  $\mathbb{P}^1\times B$.

  Let $\delta\in (\mathbb{Z}/p \mathbb{Z})^*$ be the representative
  of a prime $d$ and let  $f:X\rightarrow X$
  be a nontrivial self-map of degree $d$.
  We claim that $f$ lifts to a degree $d$ endomorphism
  $\widetilde{f}:\mathbb{P}^1\times C \rightarrow \mathbb{P}^1\times C$.
  As a first step we are going to
  show that the induced automorphism $f_{B}$ lifts to an
  automorphism $\widetilde{f}_{C}:C\rightarrow C$.
  By Proposition \ref{proposition:projciclic},
  $f_{B}^{*}(\mathcal{L})\simeq \mathcal{L}^{\pm \delta}$,
  hence  $j^{*}\circ f_{B}^{*}(\mathcal{L})\simeq j^{*}
  (\mathcal{L}^{\pm \delta})\simeq
  \mathcal{O}_C$ and there exists an induced regular map of algebraic varieties
  $\gamma :\mathbb{C}\times C \rightarrow \mathbf{L}$, where $\mathbf{L}$
  denotes the total space of the line bundle $\mathcal{L}$.
  Denote moreover by $s: C\rightarrow
  \mathbb{C}\times C$ a non zero section of $\mathbb{C}\times
  C$
  and by $\mathrm{pr}_{\mathbf{L}} :\mathbf{L} \rightarrow B$ the bundle map.

  The image $\gamma \circ s(C)$ is a projective curve in the
  complement of the zero section in $\mathbf{L}$,
  by Remark \ref{remark:torsion} the restriction of
  $\mathrm{pr}_{\mathbf{L}}$ to such a curve
  gives an \'etale covering of $B$ isomorphic to the
  \'etale covering $j:C\rightarrow B$. Denote by $\sigma$ an
  isomorphism between these two \'etale covering of $B$,
  we then have
  \begin{equation*}
    f_B\circ j= \mathrm{pr}_{\mathbf{L}} \circ \gamma \circ s=
    j  \circ \sigma \circ \gamma \circ s
  \end{equation*}
  hence $f_{B}$ lifts to
  the automorphism
  $\widetilde{f}_{C}\in \mathrm{Aut}(C)$
  given by $\widetilde{f}_{C}:=\sigma \circ \gamma \circ s$.

  Finally we are going to show that
  the existence  of the lift $\widetilde{f}_{C}$ of
  $f_{B}$ implies the existence
  of a lift $\widetilde{f}:\mathbb{P}^{1}\times C\simeq X\times_B C\rightarrow
  \mathbb{P}^{1}\times C\simeq X\times_B C$ of $f$.
  Indeed, denoting by $\hat{\pi}:X\times_B C \rightarrow
  C$ the projection on the second factor,
  we have a commutative diagram
  \begin{equation*}
    \xymatrix{ & X\times_B C \ar@{-->}[rr]^{\widetilde f}
    \ar'[d][dd]_{\hat{\pi}}
    \ar[dl]_{\overline{j}}& &
    X\times_B C \ar[dl]^{\overline{j}} \ar[dd]^{\hat{\pi}}\ \\
    X \ar[rr]^(.6){f} \ar[dd]_{\pi}& & X \ar[dd]_(.7){\pi} & \\
    & C \ar[dl]_j \ar'[r][rr]^(.3){\widetilde{f}_{C}}&
    & C \ar[dl]^j\\
    B \ar[rr]^{f_{B}} & & B &
    }
  \end{equation*}
  where the dotted arrow $\widetilde{f}$ exists by the
  universal property of the fibre product $X\times_B C$
  since $\pi\circ f\circ \overline{j}=j \circ
  \widetilde{f}_C \circ \hat{\pi}$.
  By Remark \ref{remark:action}
  we have $\widetilde{f}=\psi \times \widetilde{f}_{C}$
  where $\psi :\mathbb{P}^{1}\rightarrow \mathbb{P}^{1}$
  is a degree $d$ nontrivial self-map of $\mathbb{P}^{1}$.
  By the last part of Proposition \ref{proposition:projciclic},
  we may also suppose that either $f^{*}(S_i)=d S_i$
  or $f^{*}(S_i)=d S_j$.
  Since sections $S_1$ and $S_2$, up to a change of cooordinates,
  lift to $\{0\}\times C$ and $\{\infty \}\times C$ on
  $\mathbb{P}^1\times C$, it follows that the endomorphism
  $\psi$ is given by either
  $z\mapsto az^d$ or $z\mapsto bz^{-d}$, $a,b\in \mathbb{C}^*$.
  Since $\psi \times \widetilde{f}_{C}$ descends to $X$,
  there exists a positive integer $\delta$ such that
  for every $g\in G$
  \begin{equation}
    \label{equation:descend}
    \begin{split}
      (\psi & \times \widetilde{f}_{C})\circ \Phi_{g}=
      (\psi \times \widetilde{f}_{B})\circ (\phi_{g}\times
      g) \\
      & = (\phi_{g}^\delta \times g^\delta) \circ (\psi \times
      \widetilde{f}_{C})
      =\Phi_{g}^\delta \circ (\psi \times \widetilde{f}_{C} )\ .
    \end{split}
  \end{equation}
  Recall that
  $g\in G$ nontrivial
  acts on $\mathbb{P}^{1}\times C$ by $\phi_{g}(z,b)=(\epsilon
  z,g(b))$ where $\epsilon$ is a nontrivial $p-$th root of the unity (Remark
  \ref{remark:action}).
  The first component of equality \eqref{equation:descend} gives either
  $a\epsilon^{d}z^{d}=a\epsilon^{\delta}z^{d}$ or
  $b/\epsilon^{d}z^{d}=b\epsilon^{\delta}/z^{d}$ hence
  $\delta \equiv \pm d\ \mathrm{mod}\ p$.
  The second component of equality \eqref{equation:descend} implies
  instead
  $\widetilde{f}_{C}\circ
  g =g^{\pm \delta}\circ\widetilde{f}_{C}$ and setting
  $\varphi :=\widetilde{f}_{C}$
  we get item (ii).

  To prove the remaining implication observe that
  straightforward computations show that the degree $p$ endomorphism of
  $\mathbb{P}^{1}\times C$ given by
  $(z,b)\mapsto(\frac{1+z^{p}}{z^{p-1}},b)$ always descend to the quotient
  $X$ and the same holds for the degree $d$ endomorphism, $p\nmid d$,
  $(z,b)\mapsto
  (z^{\pm d},\varphi(b))$ if $\varphi \circ g= g^{\pm\delta}\circ \varphi$.
\end{proof}

\begin{examples}
  Fix a prime number $p$,
  let $A$ be a finite group containing a cyclic group of order $p$,
  $\mathbb{Z}/ p \mathbb{Z}\subset A$.
  Let $N_{A}(\mathbb{Z}/ p \mathbb{Z})$ be the normalizer of
  $\mathbb{Z}/ p \mathbb{Z}$ in $A$.
  Letting   $\rho_{A}:
  N_{A}(\mathbb{Z}/ p \mathbb{Z})\rightarrow \rm{Aut}(\mathbb{Z}/ p
  \mathbb{Z})$ be the omomorphism
  obtained by conjugation, we denote by
  $\overline{\rho}_{A}:N_{A}(\mathbb{Z}/ p \mathbb{Z}) \oplus
  \mathbb{Z}/2\mathbb{Z}
  \rightarrow \rm{Aut}(\mathbb{Z}/ p \mathbb{Z})$ the omomorphism given by
  $(g,\alpha)\mapsto (-1)^{\alpha}\circ \rho_{A}(g)$.
  Item (ii) of Corollary \ref{corollary:togroups} can be restated asserting
  the surjectivity of the homomorphism
  $$\overline{\rho}_{\mathrm{Aut}(C)}:
  N_{\mathrm{Aut}(C)}(\mathbb{Z}/ p \mathbb{Z}) \oplus
  \mathbb{Z}/2\mathbb{Z}
  \rightarrow \mathrm{Aut}(\mathbb{Z}/ p \mathbb{Z})\ .$$

  To construct examples of surfaces that satisfies conditions in Theorem
  \ref{theorem:projectivebundles} or Corollary \ref{corollary:togroups}
  for any prime $p$ it will
  be enough to construct a smooth projective curve $C$ of genus
  $g(C)>1$ such that its automorphisms
  group contains the group $\mathbb{Z}/ p \mathbb{Z}$  as a subgroup acting
  without fixed points and such that
  $\overline{\rho}_{\mathrm{Aut}(C)}$ is surjective.

  By a classical result of Hurwitz every finite group $\widehat G$
  can be realized as an
  automorphism group of a projective curve $C$ of genus greater
  than or equal to two,
  \cite[Corollary 3.15, p.~15]{breuer:automorphisms_rs}. Moreover we can also
  realize $\widehat G$ as a
  subgroup of $\mathrm{Aut}(C)$ in such a way that it acts freely on $C$.
  Indeed such a $\widehat G$ fits in a short exact sequence
  \begin{equation*}
    0\rightarrow K \rightarrow \Gamma \rightarrow \widehat G \rightarrow0
  \end{equation*}
  where $\Gamma$ is a free Fuchsian group that acts freely on the upper
  half plane $\mathcal{U}$ and $C \simeq \mathcal{U}/K$. But now
  $C$ is an
  intermediate covering of the universal covering $\mathcal{U}\rightarrow
  \mathcal{U}/\Gamma\simeq C/\widehat G$ and hence $\widehat G$ must act
  freely.

  Finally, choosing $\widehat G$ such that $\mathbb{Z}/ p \mathbb{Z}$ and
  $\overline{\rho}_{\widehat G}$ is surjective
  (e.g. taking a semi-direct product
  $\widehat G=\mathbb{Z}/ p \mathbb{Z}\ltimes \rm{Aut}(\mathbb{Z}/
  p \mathbb{Z})$) and letting $C$
  be a curve such that $\widehat G$ acts freely on $C$,
  we get our examples.
  \label{examples:examplescase2}
\end{examples}

\section{$K(X)=-\infty$: $\mathbb{P}^1$-bundle over an elliptic curve}
In this section, as the title suggests, we will fix our attention on $
\mathbb{P}^1$-bundles over an elliptic curve, namely case (v) of Theorem
\ref{theorem:classification}. We keep notations introduced at the beginning of
the previous section, in particular $\mathcal{E}$ will denote a locally free
sheaf of rank two on an elliptic curve $E$.

\begin{remark}
  \label{remark:possiblebundles}
  According to the classification of vector bundles of rank two over an
  elliptic curve, we may, and will, assume that $X$ is isomorphic to
  $\mathbb{P}(\mathcal{E})$
  where $\mathcal{E}$ is one of the following locally free sheaves:
  \begin{itemize}
    \item[(i)] $\mathcal{E}=\mathcal{O}\oplus \mathcal{L}$ for a line bundle
      $\mathcal{L}$ over $E$;
    \item[(ii)] there is a nontrivial extension
      \begin{equation*}
	0\rightarrow \mathcal{O}\rightarrow \mathcal{E} \rightarrow
	\mathcal{O}\rightarrow 0;
      \end{equation*}
    \item[(iii)] there exists a point $q\in E$ and a nontrivial extension
      \begin{equation*}
	0\rightarrow \mathcal{O}\rightarrow \mathcal{E}\rightarrow
	\mathcal{O}(q)\rightarrow 0.
      \end{equation*}
  \end{itemize}
\end{remark}

We start by analyzing (ii)--(iii) of Remark \ref{remark:possiblebundles}.

\begin{proposition}
  \label{proposition:elliptic2nd}
  Suppose that
  \begin{equation} \label{equation:nontriext}
    0 \rightarrow \mathcal{O}\rightarrow \mathcal{E} \rightarrow
    \mathcal{O} \rightarrow 0
  \end{equation}
  is a nontrivial extension on the curve $E$. If
  $f:\mathbb{P}(\mathcal{E})\rightarrow \mathbb{P}(\mathcal{E})$ is a
  nontrivial
  self-map then $f$ is not an automorphism of $E$.
  It follows that there are an infinite number of primes that do not appear as
  degree of a nontrivial self-map of $\mathbb{P}(\mathcal{E})$.
\end{proposition}

\begin{proof}
  Denote by $S$ the section of $X$ associated to the inclusion of
  $\mathcal{O}$ in $\mathcal{E}$ given by
  \eqref{equation:nontriext}.
  In order to prove the Proposition it is enough to prove that $S$ is the only
  irreducible curve of zero self-intersection on $X$ that is not contained in a
  fibre of the projection to $E$.
  Indeed suppose that $f$ it is an automorphism of $E$,
  since $R_f^2=0$ and on $X$ there are no curves of negative 
  self-intersection, 
  the support of $R_f$ is union of irreducible curves of
  zero self-intersection 
  whose projection to $E$ is dominant.
  If $S$ is the only curve on $X$ of such type, $R_f=mS$ for a suitable integer
  $m$, but this implies that the restriction of $f$ to the fibres must be a
  nontrivial self-map
  of $\mathbb{P}^1$ whose ramification divisor is supported on only one
  point, a
  contradiction.

  Let us prove now that $S$ is the only irreducible 
  curve of zero self-intersection 
  on $X$ that dominates $E$.
  Suppose that $S^\prime$ in another such curve. Since
  $(K_{X/E}+S^\prime).S^\prime=0$, $S^\prime$ is smooth and
  $h:S^\prime\rightarrow E$
  the restriction of
  the projection $\pi: X\rightarrow E$ to $S^\prime$ is \'etale.
  After extending the base to $S^\prime$ we have
  \begin{equation*}
    S^\prime \times_E X=\mathbb{P}(h^*\mathcal{E}),
  \end{equation*}
  and $\mathbb{P}(h^*\mathcal{E})$ has two sections of self-intersection
  zero, namely the pullback of $S$ and $S^\prime$ by the second projection. It
  follows that $h^* \mathcal{E}$ spits, and the pull-back of the extension
  \eqref{equation:nontriext}
  \begin{equation*}
    0 \rightarrow \mathcal{O}_{S^\prime}\rightarrow h^*\mathcal{E} \rightarrow
    \mathcal{O}_{S^\prime} \rightarrow 0
  \end{equation*}
  is trivial. But the map induced by $h$ on $\mathrm{Ext}^1$
  \begin{equation*}
    \mathrm{Ext}^1_E( \mathcal{O}, \mathcal{O})\rightarrow
    \mathrm{Ext}^1_{S^\prime}(\mathcal{O},\mathcal{O})
  \end{equation*}
  corresponds in co-homology to
  \begin{equation*}
    \mathrm{H}^1(E,\mathcal{O}) \rightarrow \mathrm{H}^1(S^\prime,\mathcal{O})
  \end{equation*}
  which is injective, a contradiction.
\end{proof}

\begin{proposition}
  \label{proposition:elliptic3rd}
  Let
  \begin{equation} \label{equation:nontriext2}
    0\rightarrow \mathcal{O}\rightarrow \mathcal{E} \rightarrow
    \mathcal{O}(q)\rightarrow 0
  \end{equation}
  a nontrivial extension on the elliptic curve $E$. If
  $f:\mathbb{P}(\mathcal{E})\rightarrow \mathbb{P}(\mathcal{E})$ is a
  nontrivial self-map then $\deg f \ne 2$.
\end{proposition}

\begin{proof}
  Suppose $f$ is a nontrivial self-map of degree two.
  Let $S$ be the section of $X$ corresponding to the inclusion of
  $\mathcal{O}$ in $\mathcal{E}$ given by
  \eqref{equation:nontriext2} and as usual
  let $F$ be a fibre of the projection of $\pi:X\rightarrow E$.
  Since by hypothesis $f$ has degree two and $S^2=1$ we have
  \begin{equation}
    \label{equation:auto2}
    f^*S\cdot f^*S=2\ .
  \end{equation}
  Moreover by Projection Formula
  \begin{equation}
    \label{equation:pullback2}
    F\cdot f^*S=
    \begin{cases}
      2& \textrm{if}\ \deg(f_E)=1 \\
      1& \textrm{if}\ \deg(f_E)\ne 1
    \end{cases}
  \end{equation}
  then
  \begin{equation*}
    (f^*S)^2=
    \begin{cases}
      1+2b \\
      4+4b
    \end{cases}
  \end{equation*}
  for a suitable integer $b$. In either cases this leads to a contradiction in
  view of \eqref{equation:auto2}.
\end{proof}

Now we come to (i) of Remark \ref{remark:possiblebundles}. We will need the
following elementary Lemma.
If $\deg (\mathcal{L})>0$ then $X=\mathbb{P}(\mathcal{O} \oplus
\mathcal{L})$ has a unique curve of negative self-intersection and by Lemma
\ref{lemma:negativecurve} every nontrivial self-map of $X$ has degree a
square.
Hence we may suppose that $\deg(\mathcal{L})=0$.

\begin{lemma}
  \label{lemma:autofibers}
  Suppose $X=\mathbb{P}(\mathcal{O}\oplus \mathcal{L})$, where
  $\mathcal{L}$ is a line bundle of degree zero on $E$.
  \begin{itemize}
    \item[(i)] If $ \mathcal{L}$ is not a torsion line bundle, then there is no
      nontrivial self-map of $X$ that induces an automorphism on the curve
      $E$.
    \item[(ii)] There exists a nontrivial self-map
      $f:X\rightarrow X$ that induces isomorphisms on the fibres and a
      nontrivial
      self-map $\varphi$ on the base
      if and only if
      $\varphi^*\mathcal{L}=\mathcal{L}^{\pm 1}$.
  \end{itemize}
\end{lemma}

\begin{proof}
  On $X$ there are only two
  sections of zero self-intersection, and we denote them by $S_1$ and $S_2$,
  and keep notations introduced in Remark \ref{remark:normals}.
  If
  $C$ is another irreducible curve on $X$ that dominates $E$ and such that
  $C^2=0$ and $C\ne S_1,S_2$ then since the numerical class of $C$ must be a
  multiple of the numerical class of $S_1$ we have $C\cdot S_i=0$ and $C$ is
  disjoint from $S_1,S_2$.
  Moreover the restriction of the projection $\pi :
  \mathbb{P}(\mathcal{ \mathcal{O}\oplus \mathcal{L}})\rightarrow E$,
  $\pi_{|C}:C \rightarrow E$ is \'etale.
  Extending the base to $C$ we obtain three sections of zero self-intersection,
  and then a trivial bundle.
  But then
  \begin{equation*}
    \pi_{|C}^* \mathcal{L}= (\pi_{|C}^* \circ s_1^*) \mathcal{O}(S_1) =
    \mathcal{O}_{C}(S_1)= \mathcal{O}_C
  \end{equation*}
  and $\mathcal{L}$ must be a torsion bundle. Suppose then that $S_1$ and
  $S_2$ are the only irreducible curves of zero self-intersection
  on $X$ different from a fibre. Since on $X$ there are no curves of negative
  self-intersection,
  if a nontrivial self-map $f$ induces an automorphism on the base
  the ramification divisor is such that
  $R_f=\deg(f)S_1+\deg(f)S_2$. We have that $f(S_i)=S_j$, and then arguing
  as in the proof of Proposition \ref{proposition:projciclic},
  \begin{equation*}
    f^*_E \mathcal{N}_{S_j/X}=f_E^* s_j^* \mathcal{O}(S_j)=s_i^*f^*
    \mathcal{O}(S_j)=s_i^* \mathcal{O}\big ( (\deg(f) S_i \big )=
    (\mathcal{N}_{S_i/X})^{\deg(f)}.
  \end{equation*}
  Since $\mathcal{N}_{S_1/X}=\mathcal{N}_{S_2/X}^{\vee}=\mathcal{L}$ the above
  equalities implies that $f^* \mathcal{L}= \mathcal{L}^{\pm deg(f)}$,
  and then,
  since $\deg(f)\ge 2$, $ \mathcal{L}$ is a torsion bundle.

  We are going to prove item (ii) now. We may suppose that $\mathcal{L}$ is not
  a trivial sheaf.
  Since $f$
  restricts to an automorphism on the fibres, $f^*S_i\cdot F=1$ and
  $f^{-1}(S_i)$ is a section of zero self-intersection, $i=1,2$, hence
  $f(S_i)=S_j$. Again, it follows that
  either $f_E^*\mathcal{L}=\mathcal{L}$
  or $f_E^*\mathcal{L}=\mathcal{L}^\vee$.
  This concludes the proof of the only if
  part of our statement.
  For the if part, observe that the first projection
  \begin{equation*}
    X\times_{\varphi}E\rightarrow X
  \end{equation*}
  is a finite surjective morphism and since
  $ X\times_{\varphi}E\simeq X$ as projective bundles it
  induces a nontrivial self-map of $X$ with the desired properties.
\end{proof}

\begin{remark}
  As a consequence of
  Lemma \ref{lemma:autofibers}, if $\deg(\mathcal{L})=0$ and $\mathcal{L}$ is
  not a torsion bundle, 
  every nontrivial self-map
  of $\mathbb{P}(\mathcal{O} \oplus \mathcal{L})$ induce an isomorphism
  on fibres, and the only possible degrees of nontrivial self-maps of $X$ are
  degrees of nontrivial self-maps of $E$.
  In view of Corollary \ref{corollary:abelianvarieties} we will then assume
  through the rest of this paper that $X$ is as in case (i) of Remark
  \ref{remark:possiblebundles} and $\mathcal{L}$ is a $k$-torsion line bundle.

  If $k=1,2,3$, every prime is either a multiple of $k$ or is congruent to
  $\pm 1$ modulo $k$, and by Proposition \ref{proposition:projciclic} (ii)
  there exists a nontrivial self-map of any given degree of $X$. If $k>3$ the
  existence of nontrivial self-maps of any given degree depends on the geometry
  of $E$ and $\mathcal{L}$.
  \label{remark:elliptic123}
\end{remark}

\begin{notations}
  We will denote by $E_i=\mathbb{C}/(\mathbb{C}\oplus \mathbb{C}i)$ and
  $E_\rho=\mathbb{C}/(\mathbb{C}\oplus \mathbb{C}\rho)$ where $\rho^3=1$,
  $\rho\ne 1$. We have $$\mathbf{Aut}(E_i)=\{1,-1,i,-i\}\simeq \mathbb{Z}_4$$
  and $$\mathbf{Aut}(E_\rho)=\{\frac{1}{2}+ \frac{\sqrt{3}}{2}i, -\frac{1}{2}+
  \frac{\sqrt{3}}{2}i, -\frac{1}{2}- \frac{\sqrt{3}}{2}i, \frac{1}{2}-
  \frac{\sqrt{3}}{2}i, -1,1\}\simeq \mathbb{Z}_6$$ moreover if $E\ne
  E_i,E_\rho$
  then $$\mathbf{Aut}(E)=\{1,-1\}\simeq \mathbb{Z}_2$$
  where $\mathbf{Aut}(E)$ denotes the automorphism group for the abelian
  variety
  structure of an elliptic curve $E$.
\end{notations}

\begin{remark}
  We are going to collect in this remark, for reader's convenience, some
  elementary facts on the ring of endomorphisms of an elliptic curve.
  Let $E$ be an elliptic curve and $\mathrm{End}(E)$ its ring of endomorphisms.
  If
  $E$ does not have complex multiplication then $\mathrm{End}(E)$ is a free
  $\mathbb{Z}$-module of rank one, every endomorphisms is given by
  multiplication by $n$, a given integer,
  and the only possible degrees are squares.

  If $E$ has complex multiplication then $\mathrm{End}(E)$ is a rank two free
  $\mathbb{Z}$-moudule and
  $\mathrm{End}_{\mathbb{Q}}(E)=\mathrm{End}(E)\otimes_{\mathbb{Z}}\mathbb{Q}$
  is a complex algebraic extension of $\mathbb{Q}$ and
  $[\mathrm{End}_{\mathbb{Q}}(E):\mathbb{Q}]=2$. Moreover
  $\mathrm{End}(E)\subseteq \mathcal{O}_{\mathrm{End}_{\mathbb{Q}}(E)}$,
  i.e. every
  endomorphism of $E$ is an algebraic integer in
  $\mathrm{End}_{\mathbb{Q}}(E)$,
  and
  \begin{equation*}
    \deg (\alpha)=
    \mathrm{N}_{\mathrm{End}_{\mathbb{Q}}(E)/\mathbb{Q}}(\alpha)=
    \alpha \cdot \bar \alpha,\quad \forall \alpha \in \mathrm{End}(E)\ .
  \end{equation*}
  \label{remark:endoelliptic}
  In particular, if a prime number $p$ which is not ramified in
  $\mathrm{End}_{\mathbb{Q}}(E)$ 
  is the degree of an endomorphism of $E$, 
  it splits completely in $\mathrm{End}_{\mathbb{Q}}(E)$.
\end{remark}

\begin{lemma} \label{lemma:phireduction}
  Suppose $E\ne E_i,E_\rho$ and $X=\mathbb{P}(\mathcal{O}_E\oplus \mathcal{L})$
  with $\mathcal{L}$ of $k$-torsion. If
  $X$ has a nontrivial self-map of
  any given degree then $\varphi(k)<4$, where $\varphi$ denotes the Euler's
  totient function.
\end{lemma}

\begin{proof}
  Since we suppose $\mathbf{Aut}(E)\cong \mathbb{Z}_2$, if $f:X\rightarrow
  X$ is a nontrivial self-map of prime degree $p=\deg(f)$ and $f_E\in
  \mathbf{Aut}(E)$ then by Proposition \ref{proposition:projciclic} either $p$
  is congruent to $\pm 1$ modulo $k$ or $p=k$.	Denote by $\mathfrak{P}_1$ the
  set of primes obtained in the above way, and denote by $\mathfrak{P}_2$
  the set
  of primes $p$ such that there exists a nontrivial self-map $f:X\rightarrow
  X$ with $\deg(f_E)=p$.
  By Dirichlet Theorem the analytic density of $\mathfrak{P}_1$ is less
  than or equal to $2/\varphi (k)\le 1/2$.
  In what follows we may suppose that $E$ has complex multiplication, otherwise
  the set $\mathfrak{P}_1\cup \mathfrak{P}_2$ would have analytic density less
  than or equal to $1/2$ and our claim would be clearly true.
  The set $\mathfrak{P}_2$ is then contained,
  up to a finite set containing ramified
  primes (recall Remark \ref{remark:endoelliptic} above),
  in the set of primes that
  split completely in the complex quadratic number field
  $\mathrm{End}_{\mathbb{Q}}(E)$, denote this set by $\mathfrak{P}_2^\prime$.
  It follows by Chebotarev Density Theorem
  that the analytic density of $\mathfrak{P}_2^\prime$ is less than or equal
  to $1/2$.
  If moreover $\varphi(k)>4$ the analytic density of the set $\mathfrak{P}_1$
  is
  strictly less than $2/5$ and it follows that
  $(\mathfrak{P}_1\cup \mathfrak{P}_2)^c$ contains an infinite number
  of primes.

  Suppose then that $\varphi(k)=4$ and
  that $\mathfrak{P}_1\cup \mathfrak{P}_2$ is equal to the set of
  primes, then $\mathfrak{P}_1\cup \mathfrak{P}_2^\prime$ has analytic density
  one. We are going to show that $\mathfrak{P}_1\cap \mathfrak{P}_2^\prime$
  has non zero analytic density, and this will lead us to a contradiction.
  We may suppose that the lattice of $E$ is contained
  in $\mathbb{Q}(i \sqrt{a})$ where $a$ is a positive square free integer.
  Observe that
  \begin{equation*}
    \mathfrak{P}_2^\prime = \{p\ \mathrm{prime}| \Big( \frac{-a}{p} \Big) =1 \}
  \end{equation*}
  where $(a/p)$ denotes the Legendre symbol, see \cite{ribenboim} p.~199.
  Write $a=p_1\dots p_n$ for the prime
  decomposition of $a$. Since the Legendre symbol is multiplicative, if 
  $(\frac{-1}{p})=1$ and 
  $\big(
  \frac{p_i}{p}\big)=1$ for $i=1\dots n$ then $\big( \frac{a}{p}\big)=1$.
  If $p \equiv 1\ \mathrm{mod}\ 4p_i$ then $\big( \frac{p_i}{p}\big)=1$, indeed
  for $p_i$ odd,
  \begin{equation*}
    \big( \frac{p_i}{p} \big)=(-1)^{\frac{p-1}{2}\frac{p_i-1}{2}}
    \big( \frac{p}{p_i} \big)= \big( \frac{p}{p_i} \big)
  \end{equation*}
  by Gauss' Quadratic Reciprocity Law and
  \begin{equation*}
    \big( \frac{p}{p_i} \big)\equiv p^{\frac{p_i-1}{2}}\equiv 1\
    \mathrm{mod}\ p_i
  \end{equation*}
  by Euler's Criterion. Moreover $(\frac{-1}{2})=(\frac{2}{p})=1$ if $p\equiv 1
  \ \mathrm{mod}\ 8$.
  \cite{ribenboim} p.~65 for $p_i=2$.
  It follows that $p\in \mathfrak{P}_1\cap \mathfrak{P}_2^\prime$ 
  if $p\equiv 1\ \mathrm{mod}\ 8ak$ and the analtic density of the set of
  primes satisfying this congruence equals 
  $1/\varphi (8ak)$. The proof of the Lemma is now complete.
\end{proof}

\begin{remark}
  If $\mathcal{L}$ is a $k$-torsion line bundle, $k=4,6$, by Proposition
  \ref{proposition:projciclic} the surface $\mathbb{P}(\mathcal{O}\oplus
  \mathcal{L})$, $E\ne E_i,E_\rho$, has nontrivial self-map of any given prime
  degree $p\ne 2$ if $k=4$ and $p\ne 2,3$ if $k=6$. Moreover, if $k=4$ a
  nontrivial surjective endomorphism $f:X\rightarrow X$ with $\deg(f)=2$,
  is such that $\deg(f_E)=2$. Analogously if $k=6$ and $\deg(f)=2,3$ we must
  have $\deg(f_E)=2,3$ respectively. In general if $k\ge 4$ and
  $f:X\rightarrow
  X$ is a nontrivial self-map of degree a prime $p< k$ then we must have
  $\deg(f_E)=p$.
  \label{remark:kfour}
\end{remark}

We are now in position to conclude the analysis of case (i) Remark
\ref{remark:possiblebundles}. We will make use of the following elementary
Lemma
regarding complex multiplications. In what follows we will always suppose that
the lattice relative to a given elliptic curve is generated by $1,\tau$, where
$\tau$ is a complex number. This does not involve any loss of generality.

\begin{lemma}
  The only possible complex multiplications of degree two are
  $\pm 1 \pm i,\pm \frac{1}{2}\pm \frac{\sqrt{7}}{2}i, \pm \sqrt{2}i$.
  If an elliptic curve has complex multiplication by $\pm 1 \pm i$ then it is
  isomorphic to $E_i$. If an elliptic curve has complex multiplication by
  $\pm \frac{1}{2}+\pm \frac{\sqrt{7}}{2}i$ it is isomorphic to the curve
  relative to the lattice $<1, \frac{1}{2}+ \frac{\sqrt{7}}{2}i >$.
  \label{lemma:complexmultiplication}
\end{lemma}

\begin{proof}
  For the first part of our statement, recall that if $E$ is an elliptic
  curve with complex multiplication then $\mathrm{End}_{\mathbb{Q}}(E)
  =\mathbb{Q}(i\sqrt{a})$, $a>0$ a rational integer, and
  $\mathrm{End}(E)$ is
  contained in the ring of integers of $\mathbb{Q}(i\sqrt{a})$. A direct
  computation shows then that the only complex multiplications
  of degree two are the ones
  listed in our statement.
  If an elliptic curve $E$ has complex multiplication by either $\pm 1 \pm i$
  or $\pm \frac{1}{2}+\pm \frac{\sqrt{7}}{2}i$ then
  $\mathrm{End}_{\mathbb{Q}}(E)$
  coincides respectively with the ring of integers of
  $\mathbb{Q}(i)$, $\mathbb{Q}(i\sqrt{7})$, say $R$.
  If $\Lambda$ denotes the lattice of $E$ then $R\cdot \Lambda\subseteq
  \Lambda$ and $\Lambda$ is fractional ideal of $\mathbb{Q}(i)$,
  $\mathbb{Q}(i\sqrt{7})$ respectively. These two fields have class number
  one \cite[p.~37]{neukirch:ant}, so that, after an isomorphism $\Lambda$ 
  becomes equal to $<1,i>$, $<1,
  \frac{1}{2}+ \frac{\sqrt{7}}{2}i >$ respectively.
\end{proof}

\begin{lemma}
  If $\mathcal{L}$ is a $k$-torsion line bundle on the elliptic curve $E$, with
  $k\ge 4$, then the surface $X= \mathbb{P}( \mathcal{O} \oplus \mathcal{L})$
  admits nontrivial self-maps of any given degree if and
  only if
  \begin{itemize}
    \item $k=4$, $E$ is the elliptic curve relative to the lattice
      $<1,\frac{1}{2}+\frac{i \sqrt{7}}{2}>$, and $\mathcal{L}$  is in the
      kernel of either $\frac{3}{2}+i \frac{\sqrt{7}}{2}$ or $\frac{3}{2}-i
      \frac{\sqrt{7}}{2}$.
    \item $k=5$, $E=E_i$ and $\mathcal{L}$ is in the kernel of either 
      $2+i$ or $2-i$.
    \item $k=7$, $E=E_\rho$ and $\mathcal{L}$  is in the
      kernel of either $\frac{5}{2}+\frac{i \sqrt{3}}{2}$ or
      $\frac{5}{2}-\frac{i \sqrt{3}}{2}$.
  \end{itemize}
  \label{lemma:torsiongreater}
\end{lemma}

\begin{proof}
  Suppose for the moment that $E\ne E_i,E_\rho$, as already observed,
  see Remark
  \ref{remark:kfour}, $X$ admits a nontrivial self-map of degree two if and
  only if we can lift a nontrivial self-map, say $\alpha$ of degree two
  of $E$.
  Thanks to Lemma \ref{lemma:autofibers} this is possible if and only if
  $\alpha^* \mathcal{L}=\mathcal{L}$ or what is equivalent
  \begin{equation}
    (1-\alpha)^* \mathcal{L}= 0 \ .
    \label{equation:oneminusalpha}
  \end{equation}
  By Lemma \ref{lemma:complexmultiplication}, $\alpha=\pm \frac{1}{2}\pm
  \frac{\sqrt{7}}{2}i, \pm \sqrt{2}i$. In the latter case $(1-\alpha)$ has
  degree three\footnote{Under the identification
  between points of $E$ and line bundles of degree zero modulo linear
  equivalence $\alpha^*$ operates on points of $E$ as $\hat\alpha$ the dual
  endomorphism of $\alpha$, i.e. $\alpha\circ \hat\alpha= \hat\alpha \circ
  \alpha=[\deg(\alpha)]_E$. $\hat \alpha$ is obtained from $\alpha$ by
  complex conjugation.}  
  and \eqref{equation:oneminusalpha} implies
  either $k=3$ or $k=1$ a contradiction. 
  If $\alpha =  \frac{1}{2}\pm \frac{\sqrt{7}}{2}i$ then
  \eqref{equation:oneminusalpha} implies $k=2$ a contradiction. If
  $\alpha = - \frac{1}{2}\pm \frac{\sqrt{7}}{2}i$ then
  \eqref{equation:oneminusalpha} implies $k=4$, $\mathcal{L}$ is in the
  kernel of either $\frac{3}{2}+i \frac{\sqrt{7}}{2}$ or $\frac{3}{2}-i
  \frac{\sqrt{7}}{2}$. In either cases we are able to lift $\alpha$ and $X$
  admits a nontrivial self-map of any given degree, see Remark
  \ref{remark:kfour}.

  Suppose now that $E=E_i$. In order for $X$ to admit a nontrivial self-map of
  degree two either we are able to lift $\pm 1 \pm i$ or $i^* \mathcal{L}=
  \mathcal{L}^{\pm 2}$. Hence $k=5$, $\mathcal{L}$ is in the kernel of either
  $2+i$ or $2-i$ and by Proposition \ref{proposition:projciclic} 
  we have nontrivial 
  self-maps of any given degree.

  Finally let $E=E_\rho$. In this case the only way to obtain a nontrivial
  self-map of degree two is to lift $\psi\in \mathbf{Aut}(E)$ with $\psi^*
  \mathcal{L}= \mathcal{L}^{\pm 2}$. But this last condition implies that $k=7$
  and $\mathcal{L}$ is in the kernel of either
  $\frac{5}{2}+\frac{i \sqrt{3}}{2}$ or $\frac{5}{2}-\frac{i \sqrt{3}}{2}$. It
  follows that there is an element $\lambda \in \mathbf{Aut}(E_\rho)$ such that
  $\lambda^* \mathcal{L}=\mathcal{L}^2$ hence
  and $(\lambda)^*(\lambda)^*
  \mathcal{L}= \mathcal{L}^4$.
  By Proposition \ref{proposition:projciclic} we
  have nontrivial self-maps of any given degree.
\end{proof}

Finally,for reader's convenience, we collect results obtained in this
section in
the following Corollary. It provides the last step
in establishing Theorem \ref{theorem:principal}.

\begin{corollary}
  $X$ is $\mathbb{P}^1$-bundle over an elliptic curve $E$,
  $\mathbb{P}(\mathcal{O}\oplus \mathcal{L})$, where $\mathcal{L}$ is a
  $k$-torsion line bundle on $E$ and either $k=1,2,3$ or
  \begin{itemize}
    \item[(i)] $k=4$, $E$ is the elliptic curve relative to the lattice
      $<1,\frac{1}{2}+\frac{i \sqrt{7}}{2}>$, and $\mathcal{L}$  is in the
      kernel of either $\frac{3}{2}+i \frac{\sqrt{7}}{2}$ or $\frac{3}{2}-i
      \frac{\sqrt{7}}{2}$.
    \item[(ii)] $k=5$, $E=E_i$ and $\mathcal{L}$ is in the kernel of
      either $2+i$ or $2-i$.
    \item[(iii)] $k=7$, $E=E_\rho$ and $\mathcal{L}$  is in the
      kernel of either $\frac{5}{2}+\frac{i \sqrt{3}}{2}$ or
      $\frac{5}{2}-\frac{i \sqrt{3}}{2}$.
  \end{itemize}
\end{corollary}

\begin{proof}
  This is an immediate consequence of Remark \ref{remark:elliptic123},
  Proposition \ref{proposition:elliptic2nd}, Proposition
  \ref{proposition:elliptic3rd} and Lemma \ref{lemma:torsiongreater}.
\end{proof}

\bibliography{mybiblio}

\begin{thebibliography}{10}

\bibitem{bpv:ccs}
Wolf~P. Barth, Klaus Hulek, Chris A.~M. Peters, and Antonius Van~de Ven.
\newblock {\em Compact complex surfaces}, volume~4 of {\em Ergebnisse der
  Mathematik und ihrer Grenzgebiete. 3. Folge. A Series of Modern Surveys in
  Mathematics [Results in Mathematics and Related Areas. 3rd Series. A Series
  of Modern Surveys in Mathematics]}.
\newblock Springer-Verlag, Berlin, second edition, 2004.

\bibitem{beauville:surfaces}
Arnaud Beauville.
\newblock {\em Complex algebraic surfaces}, volume~34 of {\em London
  Mathematical Society Student Texts}.
\newblock Cambridge University Press, Cambridge, second edition, 1996.
\newblock Translated from the 1978 French original by R. Barlow, with
  assistance from N. I. Shepherd-Barron and M. Reid.

\bibitem{beauville:endomorphisms}
Arnaud Beauville.
\newblock Endomorphisms of hypersurfaces and other manifolds.
\newblock {\em Internat. Math. Res. Notices}, (1):53--58, 2001.

\bibitem{birkenhake_lange:cav}
Christina Birkenhake and Herbert Lange.
\newblock {\em Complex abelian varieties}, volume 302 of {\em Grundlehren der
  Mathematischen Wissenschaften [Fundamental Principles of Mathematical
  Sciences]}.
\newblock Springer-Verlag, Berlin, second edition, 2004.

\bibitem{breuer:automorphisms_rs}
Thomas Breuer.
\newblock {\em Characters and automorphism groups of compact {R}iemann
  surfaces}, volume 280 of {\em London Mathematical Society Lecture Note
  Series}.
\newblock Cambridge University Press, Cambridge, 2000.

\bibitem{fujimoto:endo}
Yoshio Fujimoto.
\newblock Endomorphisms of smooth projective 3-folds with non-negative
  {K}odaira dimension.
\newblock {\em Publ. Res. Inst. Math. Sci.}, 38(1):33--92, 2002.

\bibitem{lang:algebra}
Serge Lang.
\newblock {\em Algebra}, volume 211 of {\em Graduate Texts in Mathematics}.
\newblock Springer-Verlag, New York, third edition, 2002.

\bibitem{nakayama:ruled}
Noboru Nakayama.
\newblock Ruled surfaces with non-trivial surjective endomorphisms.
\newblock {\em Kyushu J. Math.}, 56(2):433--446, 2002.

\bibitem{neukirch:ant}
J{\"u}rgen Neukirch.
\newblock {\em Algebraic number theory}, volume 322 of {\em Grundlehren der
  Mathematischen Wissenschaften [Fundamental Principles of Mathematical
  Sciences]}.
\newblock Springer-Verlag, Berlin, 1999.
\newblock Translated from the 1992 German original and with a note by Norbert
  Schappacher, With a foreword by G. Harder.

\bibitem{oda:convex}
Tadao Oda.
\newblock {\em Convex bodies and algebraic geometry}, volume~15 of {\em
  Ergebnisse der Mathematik und ihrer Grenzgebiete (3) [Results in Mathematics
  and Related Areas (3)]}.
\newblock Springer-Verlag, Berlin, 1988.
\newblock An introduction to the theory of toric varieties, Translated from the
  Japanese.

\bibitem{ribenboim}
Paulo Ribenboim.
\newblock {\em Classical theory of algebraic numbers}.
\newblock Universitext. Springer-Verlag, New York, 2001.

\end{thebibliography}
\bibliographystyle{plain}
\end{document}